\documentclass[10pt]{article}
\usepackage[left=2cm,right=2cm,top=2cm,bottom=2cm]{geometry}
\usepackage{algorithm,algorithmic}

\usepackage{graphicx} 
\usepackage{citesort,cite,multirow,url,graphicx}
\usepackage[cmex10]{amsmath}
\usepackage{amssymb,amsfonts,amsthm}
\usepackage[usenames,dvipsnames]{xcolor}
\usepackage{hyperref,booktabs}

\newcommand{\iter}{i} 
\newcommand{\Xk}{\X_{\iter}}
\newcommand{\Xkk}{\X_{\iter+1}}
\newcommand{\Xold}{\X_{\iter-1}}
\newcommand{\Yk}{\Y_{\iter}}
\newcommand{\Ykk}{\Y_{\iter+1}}
\newcommand{\la}{\leftarrow}
\newcommand{\Aforward}{\texttt{A}}
    \newcommand{\adj}{*} 
\newcommand{\Aadjoint}{\texttt{At}}
\newcommand{\Aadjointadjoint}{\texttt{At}^\adj}
\newcommand{\SVDfcn}{\texttt{RandomizedSVD}} 
\newcommand{\EIGfcn}{\texttt{RandomizedEIG}} 
\newcommand{\SVDlowrank}{\texttt{FactoredSVD}}
\newcommand{\QR}{\texttt{QR}}
\newcommand{\hh}{\texttt{h}}
\newcommand{\hhAdjoint}{\texttt{h}^\adj}
\newcommand{\resid}{{\bf z}}
\newcommand{\E}{\mathbb{E}}

\usepackage[normalem]{ulem} 

\newcommand{\CC}{\mathcal{C}}

\def\RR{\mathbb{R}}

\newcommand{\obs}{\mathbf y}
\newcommand{\x}{\mathbf X}
\newcommand{\X}{\mathbf X}

\newcommand{\Y}{\mathbf Y} 

\newcommand{\linmap}{\boldsymbol{\mathcal{A}}}
\newcommand{\order}{\mathcal{O}}

\newcommand{\numsam}{p}

\newcommand{\rEst}{\hat{r}}
\newcommand{\sensing}{\linmap}
\newcommand{\noise}{\boldsymbol{\varepsilon}}

\newcommand{\vectornorm}[1]{\|#1\|}

\newcommand{\proj}{\mathcal{P}}

\DeclareMathOperator{\diag}{diag}

\newcommand{\C}{\mathbb C}

\newtheorem{corollary}{Corollary}

\newcommand{\bigO}{\mathcal O}

\newcommand{\<}{\langle}
\renewcommand{\>}{\rangle}

\newtheorem{definition}{Definition} 

\newtheorem{theorem}{Theorem}

\newcommand{\bitem}{\begin{itemize}}
\newcommand{\eitem}{\end{itemize}}

\DeclareMathOperator*{\argmin}{argmin}

\DeclareMathOperator{\rank}{rank}

\newcommand{\beqn}{\begin{equation}}
\newcommand{\eeqn}{\end{equation}}
\newcommand{\balign}{\begin{align}}
\newcommand{\ealign}{\end{align}}

\newcommand{\R}{ \mathbb{R} }     

\def\proj { \mathcal{P} } 

\DeclareMathOperator{\tr}{trace}

\DeclareMathOperator{\VEC}{vec}
\long\def\symbolfootnote[#1]#2{\begingroup%
\def\thefootnote{\fnsymbol{footnote}}\footnote[#1]{#2}\endgroup}

\begin{document}
\title{Randomized Low-Memory Singular Value Projection}

\newcommand{\vectornormbig}[1]{\big\|#1\big\|}
\newcommand{\vectornormmed}[1]{\big\|#1\big\|}
\date{May 17, 2013}

\author{Stephen Becker\thanks{stephen.becker@upmc.fr, Laboratoire JLL, UPMC Paris 6, Paris} \and 
Volkan Cevher\thanks{\{volkan.cevher, anastasios.kyrillidis\}@epfl.ch, LIONS, \'Ecole polytechnique F\'ed\'erale de Lausanne} \and
Anastasios Kyrillidis\footnotemark[2]\ \thanks{Authors are listed in alphabetical order}}
\maketitle

\begin{abstract} 
Affine rank minimization algorithms typically rely on calculating the
gradient of a data error followed by a singular value decomposition at every
iteration. Because these two steps are expensive, heuristic approximations are
often used to reduce computational burden.  To this end, we propose a recovery
scheme that merges the two steps with randomized approximations, and as a
result, operates on space proportional to the degrees of freedom in the
problem. We theoretically establish the estimation guarantees of the algorithm
as a function of approximation tolerance. 
While the theoretical approximation requirements are overly pessimistic,
we demonstrate that in practice the algorithm performs well on the quantum
tomography recovery problem. 
\end{abstract}

\section{Introduction}\label{sec:intro}
In many signal processing and machine learning applications, we are given a set of observations $ \obs \in \mathbb{R}^{\numsam} $ of a rank-$r$ matrix $ \X^\star \in  \mathbb{R}^{m\times n} $ as $ \obs = \linmap \X^\star+ \noise $ via the linear operator $ \linmap: \mathbb{R}^{m\times n} \rightarrow \mathbb{R}^\numsam $, where $ r  \ll \min\lbrace m, n \rbrace $ and $\noise \in \mathbb{R}^\numsam $ is additive noise. As a result, we are interested in the solution of 
\begin{equation}\label{eq: ARM}
	\begin{aligned}
	& \underset{\X \in \mathbb{R}^{m\times n}}{\text{minimize}}
	& & f(\X) \\
	& \text{subject to}
	& & \rank(\X) \leq r,
	\end{aligned} 
	\end{equation} 
where $ f(\X) :=  \|\obs - \linmap \X\|_2^2 $ is  the data error. While the optimization problem in \eqref{eq: ARM} is non-convex, it is possible to obtain robust recovery with provable guarantees via iterative greedy algorithms (SVP)\cite{SVP,kyrillidis2012matrix} or convex relaxations \cite{recht2010guaranteed,candes2009exact} from measurements as few as $\numsam=\bigO{\left(r(m+n-r)\right)}$.

Currently, there is a great interest in designing algorithms to handle large scale versions of \eqref{eq: ARM} and its variants. As a concrete example, consider quantum tomography (QT), where we need to recover low-rank density matrices from dimensionality reducing Pauli measurements \cite{flammia2012quantum}. In this problem, the size of these density matrices grows exponentially with the number of quantum bits.
Other collaborative filtering problems, such as the Netflix challenge, also require huge dimensional optimization. Without careful implementations or non-conventional algorithmic designs, existing algorithms quickly run into time and memory bottlenecks.

These computational difficulties typically revolve around two critical issues.
First, virtually all  recovery algorithms require calculating the gradient
$\nabla f(\X)= 2\linmap^*(\linmap(\X) - \obs)$ at an intermediate iterate $\X$,
where $\linmap^*$ is the  adjoint of $\linmap$. When the range of  $\linmap^*$
contains dense matrices, this forces algorithms to use memory proportional to
$\bigO(mn)$. Second, after the iterate is updated with the gradient, projecting
onto the low-rank space requires a partial singular value decomposition (SVD).
This is usually problematic for the initial iterations of convex algorithms,
where they may have to perform full SVD's. In contrast, greedy  algorithms
\cite{kyrillidis2012matrix} fend off the complexity of full SVD's, since they
need fixed rank projections, which can be approximated via Lanczos or
randomized SVD's  \cite{structureRandomness}. 

Algorithms that avoid these two issues do exist, such as
\cite{LMaFit,recht2011parallel,max-norm-splitting,FrankWolfeOMP}, and are
typically based on the Burer-Monteiro splitting \cite{BurerMonteiro2003}. The
main idea in Burer-Monteiro splitting is to remove the non-convex rank
constraint by directly embedding it into the objective: as opposed
to optimizing $\X$, splitting algorithms directly work with its fixed factors
${\bf UV}^T=\X$ in an alternating fashion, where ${\bf U} \in \R^{m\times
\rEst}$ and ${\bf V} \in \R^{n\times \rEst}$ for some $\rEst \ge r$.
Unfortunately, rigorous guarantees are difficult.\footnote{If  $\rEst \gtrsim
\sqrt{\numsam}$, then \cite{BurerMonteiro2003} shows their method obtains a
global solution, but this is impractical for large $\numsam$. Moreover, it is
shown that the explicit rank $\rEst$ splitting method solves a non-convex
problem that has the same local minima as \eqref{eq: ARM} (if $\rEst=r$).
However, the non-convex problems are not \emph{equivalent} (e.g.~${\bf U}={\bf
0}$, ${\bf V}={\bf 0}$ is a stationary point for the splitting problem whereas
${\bf X}={\bf 0}$ is generally not a stationary point for \eqref{eq: ARM}).
Furthermore, recovery bounds for non-convex algorithms, as in
\cite{garg2009gradient} and the present paper, are statements about a sequence
of iterates of the algorithm, and say nothing about the local minima.} 
The work~\cite{alternatingMinRIP} has shown approximation guarantees if $\linmap$
satisfies a restricted isometry property with constant $\delta_{2r} \le
\kappa^2/(100r)$ (in the noiseless case), where
$\kappa=\sigma_1(\x^\star)/\sigma_r(\x^\star)$, or $\delta_{2r} \le 1/(3200
r^2)$ for a bound independent of $\kappa$. The authors suggest that these
bounds may be tightened, and that practical performance is better than the
bound suggests.

In this paper, we merge the gradient calculation and the singular value
projection steps into one and show that this not only removes a huge
computational burden, but suffers only a minor convergence speed drawback in
practice.  Our contribution is a natural but non-trivial fusion of the Singular
Value Projection (SVP) algorithm in \cite{SVP} and the approximate projection
ideas in \cite{kyrillidis2012matrix}. The SVP algorithm is an iterative hard-thresholding
algorithm that has been considered in~\cite{SVP,goldfarb2011convergence}.
Inexact steps in SVP have been considered as a
heuristic~\cite{goldfarb2011convergence} but have not been incorporated into an
overall convergence result.\footnote{ Inexact steps are often incorporated into
analysis of algorithms for convex problems. Of particular note,
\cite{FrankWolfeOMP} allows inexact eigenvalue computations in a modified
Frank-Wolfe algorithm that has applications to \eqref{eq: ARM}.}
A non-convex framework for affine rank minimization (including variants of the SVP algorithm) that utilizes inexact projection operations with provable signal approximation and convergence guarantees
is proposed in \cite{kyrillidis2012matrix}.
Neither \cite{SVP, kyrillidis2012matrix} considers splitting
techniques in the proposed schemes.

This work, departing from \cite{SVP, kyrillidis2012matrix}, engineers the
SVP algorithm to operate like splitting algorithms that {\it directly work with
the factors}; this added twist decreases the per iteration requirements in
terms of storage and computational complexity.  Using this new formulation,
each iteration is nearly as fast as in the splitting method, hence removing a
drawback to SVP in relation to splitting methods.  Furthermore, we prove that,
under some conditions, it is still possible to obtain perfect recovery even if
the projections are inexact.  In particular, our assumption is that the linear
map $\linmap$ satisfies the rank restricted isometry property, and in
section~\ref{sec:quantum} we give an application that satisfies this
assumption, allowing perfect recovery (in the noiseless case) or stable
recovery (in the presence of noise) from measurements $\numsam \ll mn$.
 This approach has been used for convex~\cite{recht2010guaranteed}
and non-convex \cite{SVP,kyrillidis2012matrix} algorithms to obtain
approximation guarantees.

\section{Preliminary material}\label{sec:prelims}
Notation: we write $\proj_\Omega$ to be an orthogonal projection onto the closed set $\Omega$ when it exists. For shorthand we write $\proj_r$ to mean  
$\proj_{\{\x: \rank(\x)\le r\}}$ (which does exist by the Eckart-Young theorem). Computer routine names are typeset with a \texttt{typewriter font}.

\subsection{R-RIP}
The Rank Restricted Isometry Property (R-RIP) is a common tool used in matrix recovery~\cite{recht2010guaranteed,SVP,kyrillidis2012matrix}:
\begin{definition}[R-RIP for linear operators on matrices \cite{recht2010guaranteed}]{\label{def:RIP}} A linear operator $ \linmap: \RR^{m\times n} $ $\rightarrow \RR^{\numsam} $ satisfies the R-RIP with constant $ \delta_{r}(\linmap) \in (0,1)$ if, $\forall \X \in \RR^{m\times n}$ with $\rank(\X)\leq r$, 
\begin{equation}\label{eq:RIP}
  (1-\delta_{r}(\linmap))\vectornormbig{\X}_F^2 \leq \vectornormbig{\linmap \X}_2^2 \leq (1+\delta_{r}(\linmap))\vectornormbig{\X}_F^2, 
\end{equation} 
We write $\delta_{r}$ to mean $\delta_{r}(\linmap)$.
\end{definition}

\subsection{Additional convex constraints}
Consider the variant 
\begin{equation}\label{eq:ARM2}
	\begin{aligned}
	& \underset{\X \in \mathbb{R}^{m\times n}}{\text{minimize}}
	& & f(\X) \\
	& \text{subject to}
	& & \rank(\X) \leq r,\; \x\in \CC,
	\end{aligned} 
\end{equation} 
for a convex set $\CC$. Our main interests are $\CC_+=\{ \x: \x \succeq 0 \}$
and the matrix simplex $\CC_\Delta=\{ \x: \x \succeq 0,\; \tr(\x)=1 \}$. In both cases the constraints 
are unitarily invariant and the projection onto these sets can be done by taking the eigenvalue decomposition and projecting the eigenvalues. 
Furthermore, for these specific $\CC$, $\proj_{\{\x: \rank(\x)\le r\} \cap \CC } = \proj_\CC \circ \proj_r$ (this is not obvious; see ~\cite{kyrillidis2012simplexICML}).\footnote{This formula is literally true for $\CC_+$ and $\{ \x: \x \succeq 0,\; \tr(\x)\le 1 \}$. For $\CC=\{ \x: \x \succeq 0,\; \tr(\x)= 1 \}$ constraints, $\proj_\CC$ can increase the rank, so formally we must work on a restricted subspace and then embed back in the larger space, but this poses no theoretical issues.}

In general, any convex set $\CC$ satisfying the above property is compatible with our algorithm, as long as $\x^\star \in \CC$. 
We overload notation to use $\proj_\CC$ to denote both the projection of $\x$ onto the set as well as the projection of its eigenvalues onto the analogous set.

\subsection{Approximate singular value computations} 
The standard method to compute a partial SVD is the Lanczos method.
By itself it is not numerically stable and requires re-orthogonalization and implicit restarts. Excellent implementations are available, but it is a sequential algorithm that calls matrix-vector products. This makes it more difficult to parallelize, which is an issue on modern multi-processor computers. The matrix-vector multiplies are also slower than grouping into matrix-matrix multiplies since it is harder to predict memory usage and this will lead to cache misses; it also precludes the use of theoretically faster algorithms such as Strassen's. Theoretically, there are no known relative error bounds in norm (\`a la Theorem~\ref{thm:Joel}).

\begin{algorithm}[t]
{\em 
Finds $Q$ such that $X \approx \proj_Q X$ where $\proj_Q = QQ^\adj $.}
\begin{algorithmic}[1]
    \REQUIRE Function $\hh : \widetilde{Z} \mapsto X\widetilde{Z}$
    \REQUIRE Function $\hhAdjoint : \widetilde{Q} \mapsto X^\adj\widetilde{Q}$
    \REQUIRE $r \in \mathbb{N}$ \COMMENT{Rank of output}
    \REQUIRE $q \in \mathbb{N}$ \COMMENT{Number of power iterations to perform}
    \STATE $\ell = r + \rho$ \COMMENT{Typical value of $\rho$ is 5}
    \STATE $\Omega$ a $n \times \ell$ standard Gaussian matrix
    \STATE $W\leftarrow \hh(\Omega)$
    \STATE $Q \la \QR(W)$ \hfill \COMMENT{The QR algorithm to orthogonalize $W$}
    \FOR{ $j=1,2,\ldots,q$ }
      \STATE $Z \la  \QR( \hhAdjoint (Q) )$
      \STATE $Q \la  \QR( \hh(Z) )$
    \ENDFOR
    \STATE $Z \la   \hhAdjoint (Q) $
    \STATE $(U,\Sigma,V) \la \SVDlowrank(Q,I_{\ell},Z)$ \COMMENT{$\widetilde{\x}_{i+1}=U\Sigma V^\adj$ in the appendix}
    \STATE Let $\Sigma_r$ be the best rank $r$ approximation of $\Sigma$
    \RETURN $(U,\Sigma_r,V) $  \COMMENT{$\x_{i+1}=U\Sigma_r V^\adj$ in the appendix}
\end{algorithmic}
\caption{$\SVDfcn$} 
\label{algo:rankProjection}
\end{algorithm}

\begin{algorithm}[t]
    \newcommand{\UU}{\widetilde{U}}
    \newcommand{\VV}{\widetilde{V}}
    \newcommand{\DD}{\widetilde{D}}
    {\em Computes the SVD $U\Sigma V^\adj$ of the matrix $X$ implicitly given by $X=\UU\DD{\VV}^\adj$}
\begin{algorithmic}[1]
    \STATE $(U,R_U) \la \QR(\UU)$
    \STATE $(V,R_V) \la \QR(\VV)$
    \STATE $(u,\Sigma,v) \la \texttt{DenseSVD}(R_U\DD R^\adj_V)$
    \RETURN $(U,\Sigma,V) \la (Uu,\Sigma,Vv)$
\end{algorithmic}
\caption{$\SVDlowrank(\UU,\DD,\VV)$}
\label{algo:svdLowRank}
\end{algorithm}

As an alternative, we turn to randomized linear algebra.
On this front, we restrict ourselves to algorithms that require only multiplications, as opposed to sub-sampling entries/rows/columns, as sub-sampling is not efficient for the application we present. The randomized approach  presented in Algorithm~\ref{algo:rankProjection} has been rediscovered many times, but has seen a recent resurgence of interest due to theoretical analysis~\cite{structureRandomness}:

\begin{theorem}[Average Frobenius error] \label{thm:Joel}
Suppose $\X\in \R^{m\times n}$, and choose a target rank $r$ and oversampling parameter $\rho \ge 2$ where $\ell:= r+\rho \le \min\{m,n\}$. Calculate $Q$ and $\proj_Q$ via \SVDfcn\ using $q=0$
and set $\widetilde{\X} = \proj_Q \X$ (which is rank $\ell$).
Then 
$$ \E \|\X-\widetilde{\X}\|_F^2 \le \left( 1 + \epsilon\right)\|\X-\X_r\|_F^2 $$
where $\X_r$ is the best rank $r$ approximation 
in the Frobenius norm 
of $\X$ and $\epsilon = \frac{r}{\rho-1}$.
\end{theorem} The theorem follows from the proof of Thm.~10.5 in \cite{structureRandomness} (note that Thm.~10.5 is stated in terms of $ \E \|\X-\widetilde{\X}\|_F $ which is not the same as  $ \sqrt{\E \|\X-\widetilde{\X}\|_F^2 }$ ). The expectation is with respect to the Gaussian r.v. in \SVDfcn. For the sake of our analysis, we cannot immediately truncate $\widetilde{\X}$ to rank $r$ since then the error bound in \cite{structureRandomness} is not tight enough. Thus, since $\widetilde{X}$ is rank $\ell$, in practice we even observe that 
$ \|\X-\widetilde{\X}\|_F^2 < \|\X-\X_r\|_F^2$, especially for small $r$, as shown in Figure~\ref{fig:epsilon}. The figure also shows that using $q > 0$ power iterations is extremely helpful, though this is not taken into account in our analysis since there are no useful theoretical bounds (in the Frobenius norm). Note that variants for eigenvalues also exist; we refer to the equivalent of \SVDfcn\  as \EIGfcn, which has the property that $U=V$ and $\Sigma$ need not be positive (cf., \cite{structureRandomness,GittensMahoney2013})

\section{Algorithm} \label{sec:algorithm}

\subsection{Projected gradient descent}
Our approach is based on the projected gradient descent algorithm: 
\begin{equation}\label{eq: projected gradient}
  \Xkk = \proj_r^\epsilon(\Xkk - \mu_\iter  \nabla f(\Xk)),
\end{equation}
where $\Xk$ is the $\iter$-th iterate, $\nabla f(\cdot)$ is the gradient of the loss function, $\mu_\iter$ is a step-size, and $\proj_r^\epsilon(\cdot)$ is the approximate projector onto rank $r$ matrices given by \SVDfcn.  If we include a convex constraint $\CC$, then the iteration is 
\begin{equation}\label{eq: projected gradientC}
  \Xkk = \proj_\CC(\proj_r^\epsilon(\Xkk - \mu_\iter  \nabla f(\Xk))).
\end{equation}

In practice, Nesterov acceleration improves performance:
\begin{align}\label{eq:Nesterov}
  \Ykk      &= (1+\beta_\iter) \Xk - \beta_\iter \Xold \\
  \Xkk      &= \mathcal{P}(\Yk - \mu_\iter \nabla f(\Yk)),
\end{align}
where $\beta_\iter$ is chosen $\beta_\iter = (\alpha_{\iter-1} -1)/\alpha_\iter$ and $\alpha_0=1$, 
$ 2 \alpha_{\iter+1} = 1 + \sqrt{4 \alpha_\iter^2 + 1 }$~\cite{Nesterov83} (see \cite{kyrillidis2012matrix}).
 Theorem~\ref{thm: invariant} holds for a stepsize $\mu_i$ based on the RIP constant, which is unknown. In practice, the algorithm consistently converges as long as $\mu_i \lesssim \frac{2}{\|\linmap\|^2}$.

\begin{algorithm}[t]
    \newcommand{\dk}{d_\iter}
    \newcommand{\dkk}{d_{\iter+1}}
    \newcommand{\dold}{d_{\iter-1}}
    \newcommand{\uk}{u_\iter}
    \newcommand{\ukk}{u_{\iter+1}}
    \newcommand{\uold}{u_{\iter-1}}
    \newcommand{\vk}{v_\iter}
    \newcommand{\vkk}{v_{\iter+1}}
    \newcommand{\vold}{v_{\iter-1}}
    
\begin{algorithmic}[1]
    \REQUIRE step-size $\mu > 0$, measurements $\obs$, initial points $u_0\in \mathcal{K}^{m \times r},\; v_0\in\mathcal{K}^{n\times r},\; d_0\in\mathcal{K}^{r}$
    \REQUIRE (optional) unitarily invariant convex set $\CC$
 \REQUIRE Function $\Aforward:(u,d,v) \mapsto \linmap(u\diag(d)v^\adj)$
 \REQUIRE Function $\Aadjoint:(\resid,w) \mapsto \linmap^*(\resid)w$
 \REQUIRE Function $\Aadjointadjoint:(\resid,w) \mapsto (\linmap^*(\resid))^\adj w$
 \STATE $v_{-1}\leftarrow 0$, $u_{-1}\la 0$, $d_{-1}\la 0$
 \FOR{ $\iter=0,1,\ldots$ }
  \STATE Compute $\beta_\iter$ \COMMENT{See text}
  \STATE $u_y \la [\uk,\uold]$, $v_y \la [\vk,\vold]$
  \STATE $d_y \la [(1+\beta_\iter)\dk,-\beta_\iter \dold]$
  \STATE $\resid \la A(u_y,d_y,v_y) - \obs$ \hfill \COMMENT{Compute the residual}
  \STATE Define the functions \\ $\hh: w \mapsto u_y\diag(d_y)v_y^\adj w - \mu \Aadjoint(\resid,w)$ \\
$\hhAdjoint: w \mapsto v_y\diag(d_y)u_y^\adj w - \mu \Aadjointadjoint(\resid,w)$
  \STATE $(\ukk,\dkk,\vkk) \la \SVDfcn(\hh,\hhAdjoint,r)$ or $(\ukk,\dkk,\ukk) \la \EIGfcn(\hh,\hhAdjoint,r)$
  \STATE $\dkk \la \proj_\CC( \dkk) $ \COMMENT{Optional} \label{algo:optional}
 \ENDFOR
 \RETURN $X\leftarrow \uk\dk\vk^\adj$ \COMMENT{If desired}
\end{algorithmic}
\caption{Efficient implementation of SVP, $\mathcal{K}=\{\R,\C\}$}
\label{algo:1}
\end{algorithm}

Algorithm~\ref{algo:1} shows implementation details that are important for keeping low-memory requirements. The implementation of maps like $\Aforward$ and $\Aadjoint$ depends on the structure of $\linmap$; see section \ref{sec:quantum} for explicit examples.

\section{Convergence} \label{sec:convergence}

We assume the observations are generated by $ \obs = \linmap \X^\star+ \noise $ where $\noise$ is a noise term, not to be confused with the approximation error $\epsilon$. In the following theorem, we will assume that $\|\linmap\|^2 \le mn/\numsam $, which is true for the quantum tomography example~\cite{liu2011universal}; if $\linmap$ is a normalized Gaussian, then this assumption holds in expectation.
\begin{theorem}(Iteration invariant)\label{thm: invariant}  
Pick an accuracy $\epsilon = \frac{r}{\rho-1} $, where $\rho$ is defined as in Theorem 1. 
Define $\ell=r+\rho$ and let $c$ be an integer such that $\ell = (c-1)r$.
Let $\mu_i = \frac{1}{2(1+\delta_{cr})}$ in \eqref{eq: projected gradient} and assume
$\|\linmap\|^2 \le mn/\numsam$
and $f(\X_i)>C^2\|\noise\|^2$, where $C \ge 4$ is a constant. Then the descent scheme \eqref{eq: projected gradient} or 
\eqref{eq: projected gradientC}
has the following iteration invariant
  \begin{equation}\label{eq:iter}
    \E f(\X_{i+1}) \le \theta  f(\X_{i}) + \tau \|\noise\|^2, 
  \end{equation}
  in expectation, where 
\begin{equation*}
    \theta \le 12 \cdot \frac{1+\delta_{2r}}{1 - \delta_{cr}} \cdot 
\left( 
    \frac{\epsilon}{1+\delta_{cr}}\cdot\frac{m n }{\numsam} + (1+\epsilon)\frac{3\delta_{cr}}{1-\delta_{2r}} 
\right),
\end{equation*}
and 
\begin{equation*}
\tau \le  
  \frac{1+\delta_{2r}}{1 - \delta_{cr}}  \cdot  
 \left( 
 12\cdot(1+\epsilon)  \left(1 + \frac{2\delta_{cr}}{1-\delta_{2r}}\right) 
 + 8
 \right). 
\end{equation*}
  The expectation is taken with respect to Gaussian random designs in \SVDfcn.
  If $\theta \le \theta_\infty < 1$ for all iterations, then $\lim_{i \rightarrow \infty} \E f(\X_{i}) \le \max\{C^2,\frac{\tau}{1-\theta_\infty} \}\|\noise\|^2$.
\end{theorem}
Each call to \SVDfcn\  draws a new Gaussian r.v., so the expected value does not depend on previous iterations.
By Corollary 3.4 in~\cite{cosamp}, $\delta_{cr} \le c \cdot \delta_{2r}$, which allows us to put $\theta$ and $\tau$ in terms of $\delta_{2r}$ if desired, at a slight expense in sharpness.

The expected value of the function converges linearly at rate $\theta$ to within a constant of the noise level, and in particular, it converges to zero when there is no noise since $C$ and $\tau$ are finite. 
Note that convergence of the iterates follows from convergence of the function $f$:
\begin{corollary}  
    If $f(\X_i) \le \gamma$, then $\|\X_i - \X^\star\|_F^2 \le \frac{(\sqrt{\gamma}+\|\noise\|_2 )^2}{1-\delta_{2r}}$.
\end{corollary}
\begin{proof}
By the R-RIP and the triangle inequality,
\begin{align*}
    \sqrt{1+\delta_{2r}(\linmap)}\|\X_i - \X^\star\|_F & \le \| \linmap(\X_i) - \linmap(\X^\star) \|_2 \\
    & = \| (\linmap(\X_i)-\obs) - (\linmap(\X^\star)-\obs) \|_2 \\
    & \le \| (\linmap(\X_i)-\obs) \|_2 + \|\noise \|_2 \\
    & \le \sqrt{\gamma} + \|\noise \|_2
\end{align*}
\end{proof}

\begin{corollary}[Exact computation]\label{cor:exact} %
   If $\epsilon=0$ and there is no additional convex constraint $\CC$, then $\theta=\frac{2\delta_{2r}}{  1 - \delta_{2r} }( 1+\frac{2}{C})$ and $\tau = 1+\frac{2\delta_{2r}}{1-\delta_{2r}}$,
    hence $\theta < 1$ if $\delta_{2r} < \frac{1}{3+4/C}$.
\end{corollary}
Corollary~\ref{cor:exact} shows that without the approximate SVD, the R-RIP constants are quite reasonable. 
For example, with exact computation and no noise, any value of $\delta_{2r} < 1/3$ implies that $\lim_{i \rightarrow \infty} \X_i = \X^\star$. With noise, choosing $C=4$ gives $\delta_{2r} = 1/5$ and 
$\theta=3/4$, $\tau=3/2$ and thus $\lim_{i \rightarrow \infty} f(\X_i) \le \max\{ 16, 6 \}\|\noise\|^2$.

Note that the theorem gives pessimistic values for $\epsilon$. We want the bound on $\theta$ to be less than $1$ in order to have a contraction, so we need 
\begin{equation*}
     \underbrace{12 \cdot \frac{1+\delta_{2r}}{1 - \delta_{cr}} \cdot 
    \frac{\epsilon}{1+\delta_{cr}}\cdot\frac{m n }{\numsam}}_{\text{I}}
    + 
    \underbrace{12(1+\epsilon) \cdot \frac{1+\delta_{2r}}{1 - \delta_{cr}} \cdot 
    \frac{3\delta_{cr}}{1-\delta_{2r}}}_{\text{II}}
    < 1
\end{equation*}
For a rough analysis, we will give approximate conditions so that each of the I and II terms is less than $0.5$. 
It is clear that the terms blow up if $\delta_{cr} \rightarrow 1$, so we will assume $\delta_{cr} \ll 1$ (and hence $\delta_{2r}\ll 1$). Then setting $1+\delta_{2r}\approx 1$ in the numerator of I, we require that
\begin{equation}
    \frac{12}{1-\delta_{cr}^2}\cdot\frac{\epsilon\, mn}{p} < \frac{1}{2}
\end{equation}
which means that 
we need $\epsilon \lesssim \frac{p}{24mn}$. For quantum tomography, $m=n$ and $p=\order(rn)$, so we require $\epsilon \lesssim \mathcal{O}(r/n)$. From Theorem~\ref{thm:Joel}, our bound on $\epsilon$ is $r/(\rho-1)$, so we require $\rho \simeq n$, which defeats the purpose of the randomized algorithm (in this case, one would just do a dense SVD).
Numerical examples in the next section will show that $\rho$ can be nearly a small constant, so the theory is not sharp.

For the II term, again approximate $1+\delta_{2r}\approx 1$ and then multiply the denominators and ignore the $\delta_{cr}\delta_{2r}$ term to get
\begin{equation}
    72 \delta_{cr}(1+\epsilon) \lesssim 1-\delta_{2r}-\delta_{cr}. 
\end{equation}
Since certainly $\epsilon \le 0.5$ and $\delta_{2r}+\delta_{cr}\le 0.5$, a sufficient condition is $\delta_{cr} < 1/216$, which is reasonable (cf.~\cite{alternatingMinRIP}).

\section{Numerical experiments} \label{sec:numerics}

\newcommand{\qbit}{q_b}

\subsection{Application: quantum tomography}\label{sec:quantum}
As a concrete example, we apply the algorithm to the quantum tomography problem, which is a particular instance of \eqref{eq: ARM}. For details, we refer to \cite{QuantumTomoPRL,flammia2012quantum}. The salient features are that the variable $\X\in\C^{n \times n}$ is constrained to be Hermitian positive-definite, and that, unlike many low-rank recovery problems, the 
linear operator $\linmap$ satisfies the R-RIP: \cite{liu2011universal} establishes that Pauli measurements (which comprise $\linmap$) have R-RIP with overwhelming probability when $\numsam=\bigO{(rn\log^6n)}$. 
In the ideal case, $\x^\star$ is exactly rank $1$, but it may have larger rank due to some (non-Gaussian) noise processes, in addition to AWGN $\noise$. 
Furthermore, it is known that the true solution $\X^\star$ has trace 1, which is also possible to exploit in our algorithmic framework.

Since $\X$ is Hermitian, 
the $u$ and $v$ terms in the algorithm are identical. Several computations can be simplified and there is a version of Algorithm~\ref{algo:rankProjection} which exploits the positive-definiteness to incorporate a Nystr\"om approximation (and also forces the approximation to be positive-definite); see \cite{structureRandomness,GittensMahoney2013}. Here, we focus on showing how the functions $\Aforward$ and $\Aadjoint$ can be computed (due to the complex symmetry, $\Aadjointadjoint=\Aadjoint$).

In quantum tomography, the linear operator has the form
$(\linmap(\X))_j = \<{\bf E}_j,\X\>$ where ${\bf E}_j={\bf E}_j^\adj$ is the Kronecker product of $2 \times 2$
Pauli matrices. There are four possible Pauli matrices $\sigma_{x,y,z}$ if we define $\sigma_I$ to be the $2\times 2$ identity matrix. For a $\qbit$-qubit system, 
$ {\bf E}_j = \sigma_{j1} \otimes \sigma_{j2} \otimes \ldots \otimes \sigma_{j\qbit}$.
For roughly 12 qubits and fewer, it is simple to calculate $\linmap(\X)$ by 
explicitly forming ${\bf E}_j$ and then creating a sparse matrix ${\bf A}$ with the $j^\text{th}$ row of ${\bf A}$ equal to $\VEC({\bf E}_j)$ so that $\linmap(\X) = {\bf A} \VEC(\X)$. For larger systems,
storing this sparse matrix is impractical since there are $\numsam \ge n$ rows and each row has exactly $n$ non-zero entries, so there are over $n^2$ entries in ${\bf A}$. 

To keep memory low, we exploit the Kronecker-product nature of ${\bf E}_j$
and store it with only $\qbit$ numbers.  
When $\X={\bf xx}^\adj$, we compute $ \<{\bf E}_j,\X\> = \tr({\bf E}_j{\bf xx}^\adj)
= \tr({\bf x}^\adj {\bf E}_j{\bf x})$, and ${\bf E}_j{\bf x}$ can be computed in $\order(\qbit n)$ time. This gives us $\Aforward$. The output of $\Aforward$ is real even when $\X$ is complex.

To compute $\Aadjoint(\resid,{\bf w})$ when the dimensions are small, we just explicitly form the matrix ${\bf M} = \linmap(\resid)$ and then multiply ${\bf Mw}$. To form ${\bf M}$, we use the same sparse matrix ${\bf A}$ as above and reshape the $n^2$ vector ${\bf A}^*\resid$ into a $n \times n$ matrix.
For larger dimensions, when it is impractical to store ${\bf A}$, we implicitly represent
$ {\bf M} = \sum_{j=1}^\numsam \resid_j {\bf E}_j$ and thus ${\bf Mw} = \sum_{j=1}^\numsam \resid_j {\bf E}_j {\bf w}$. In general, the output is complex. However, if it is known \emph{a priori} that $\X$ is real-valued, this can be exploited by taking the real part of ${\bf M}$. This leads to a considerable time savings ($2\times$ to $4\times$), and all experiments shown below make this assumption.

In our numerical implementation, we code both $\Aforward$ and $\Aadjoint$ in C and 
parallelize the code since this is the most computationally expensive calculation. Our parallelization implementation 
uses both \texttt{pthreads} on local cores as well as message passing among different computers.
There are two approaches to parallelization: divide the indices $j=1,\ldots,\numsam$ among different cores, or, when ${\bf x}$ or ${\bf w}$ has several columns, send different columns to the different cores. Both approaches are efficient in terms of message passing since $\linmap$ is parameterized and static. The latter approach only works when ${\bf x}$ or ${\bf w}$ has a significant number of columns, and so it does not apply to Lanczos methods that perform only matrix-vector multiplies.

Recording error metrics can be costly if not done correctly. 
Let  
$\X={\bf xx}^\adj$ and ${\bf Y}={\bf yy}^\adj$ be rank-$r$ factorizations.
For the Frobenius norm error 
$\|\X-{\bf Y}\|_F$ 
which requires $n^2$ operations naively, we expand the term and use the cyclic invariance of trace to get $\|\X-{\bf Y}\|_F^2= \tr({\bf x}^\adj {\bf xx}^\adj {\bf x}) + \tr({\bf y}^\adj {\bf yy}^\adj {\bf y}) - 2\tr({\bf x}^\adj {\bf yy}^\adj {\bf x}) $, which requires only $\order(nr^2)$ flops.
In quantum information, another common metric is the trace distance~\cite{nielsen2010quantum} 
$ \|\X-{\bf Y}\|_*$, 
where $\|\cdot\|_*$ is the nuclear norm. This calculation  requires $\order(n^3)$ flops if calculated directly but can also be calculated cheaply via $\SVDlowrank$ on ${\bf U}={\bf V}=[{\bf x},{\bf y}]$ and ${\bf D}=[\mathbb{I},{\bf 0};{\bf 0},-\mathbb{I}]$.
\newcommand{\hf}[1]{#1^{1/2}}
The third common metric is the fidelity~\cite{nielsen2010quantum} 
given by $\| \hf{\X} \hf{{\bf Y}} \|_*$.  If either $\X$ or ${\bf Y}$ is rank-1, this can be calculated cheaply as well. 

\subsection{Results}

\begin{figure}[ht]
    \centering
    \includegraphics[width=.42\columnwidth]{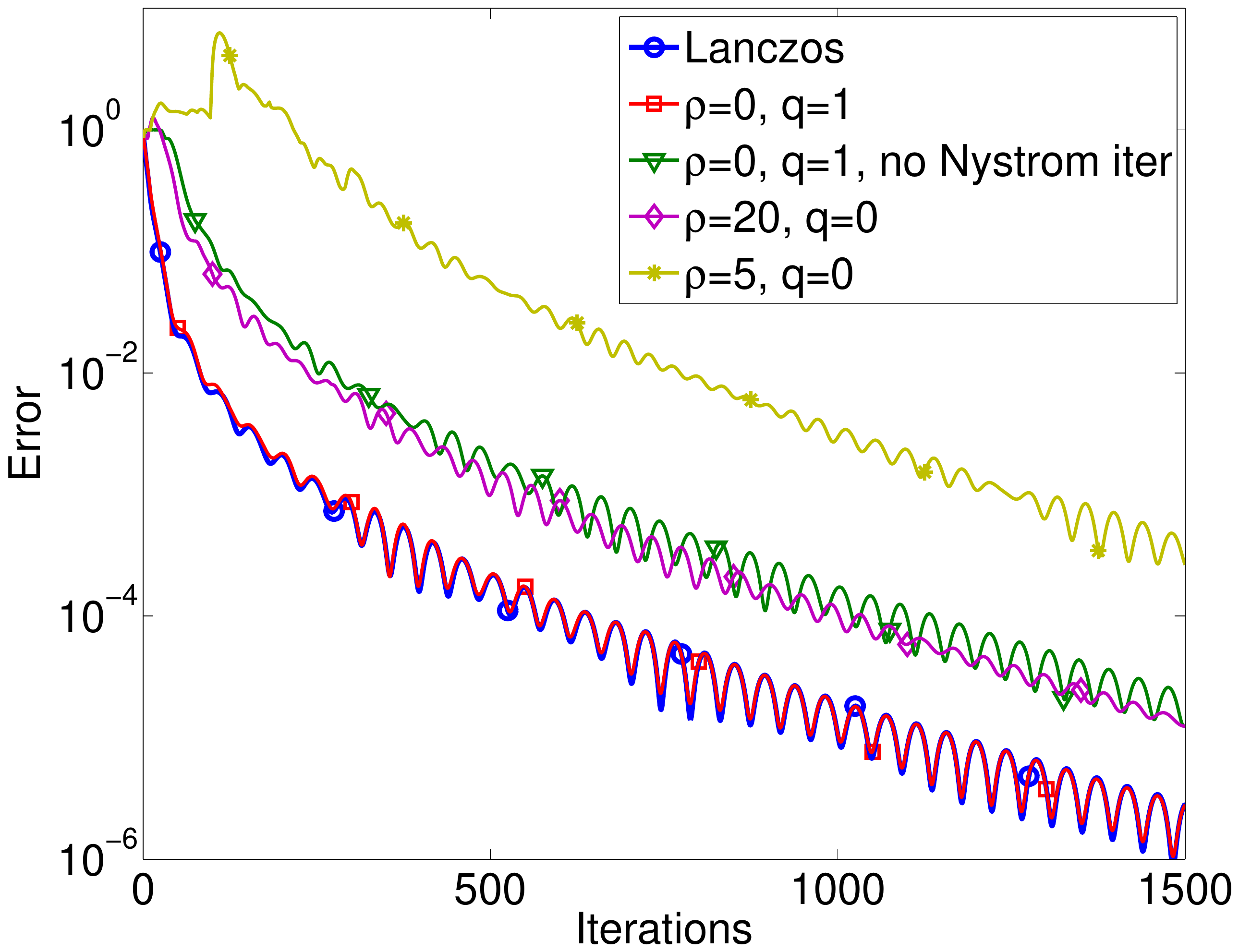} 
    \includegraphics[width=.42\columnwidth]{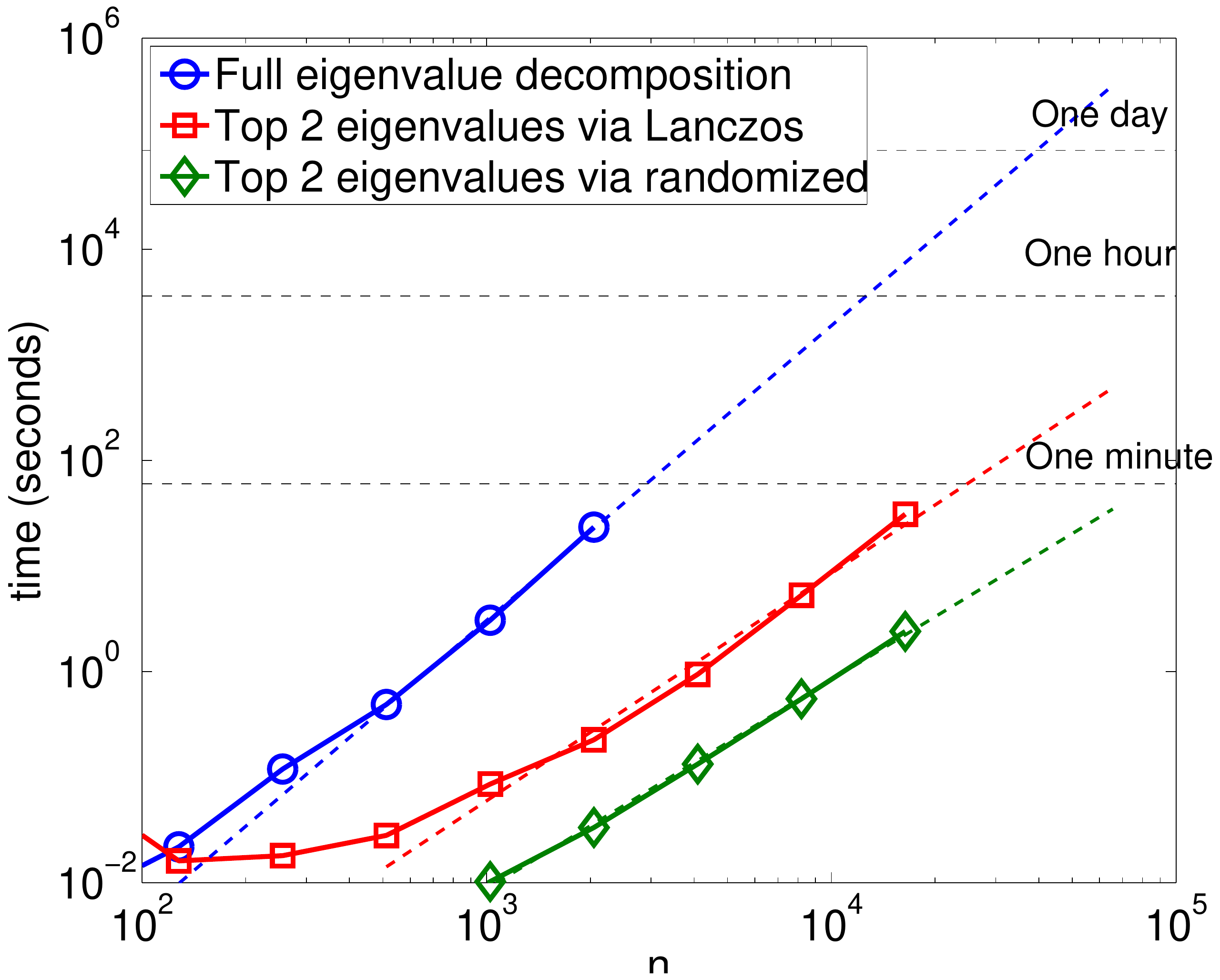}
    \caption{(Left) Convergence rate as a function of parameters to \SVDfcn/\EIGfcn. 
    (Right) Comparison of just eigenvalue computation times via three methods.
    }
    \label{fig:convergence}
\end{figure}

\begin{figure}[ht]
    \centering
    \includegraphics[width=.42\columnwidth]{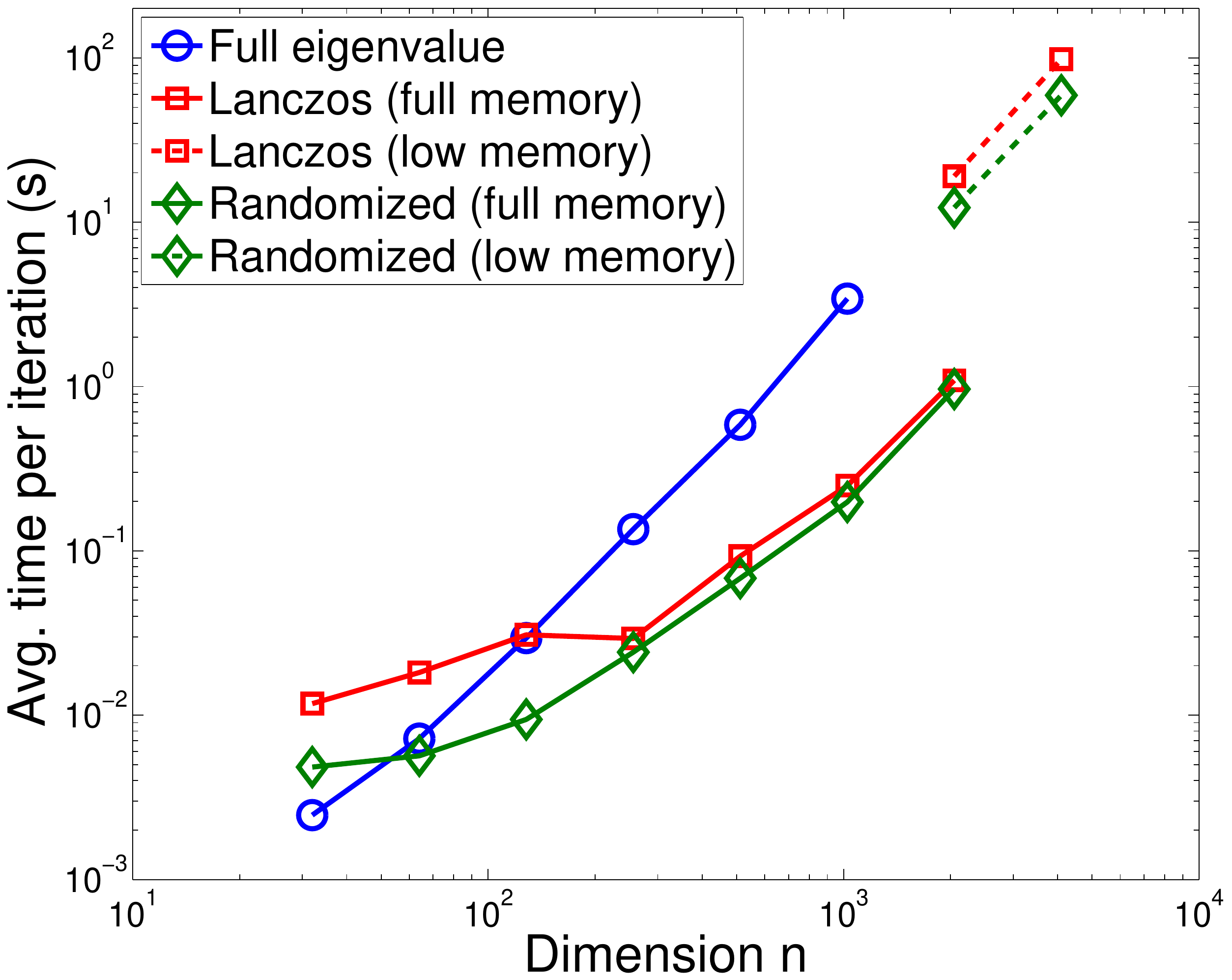} 
    \includegraphics[width=.42\columnwidth]{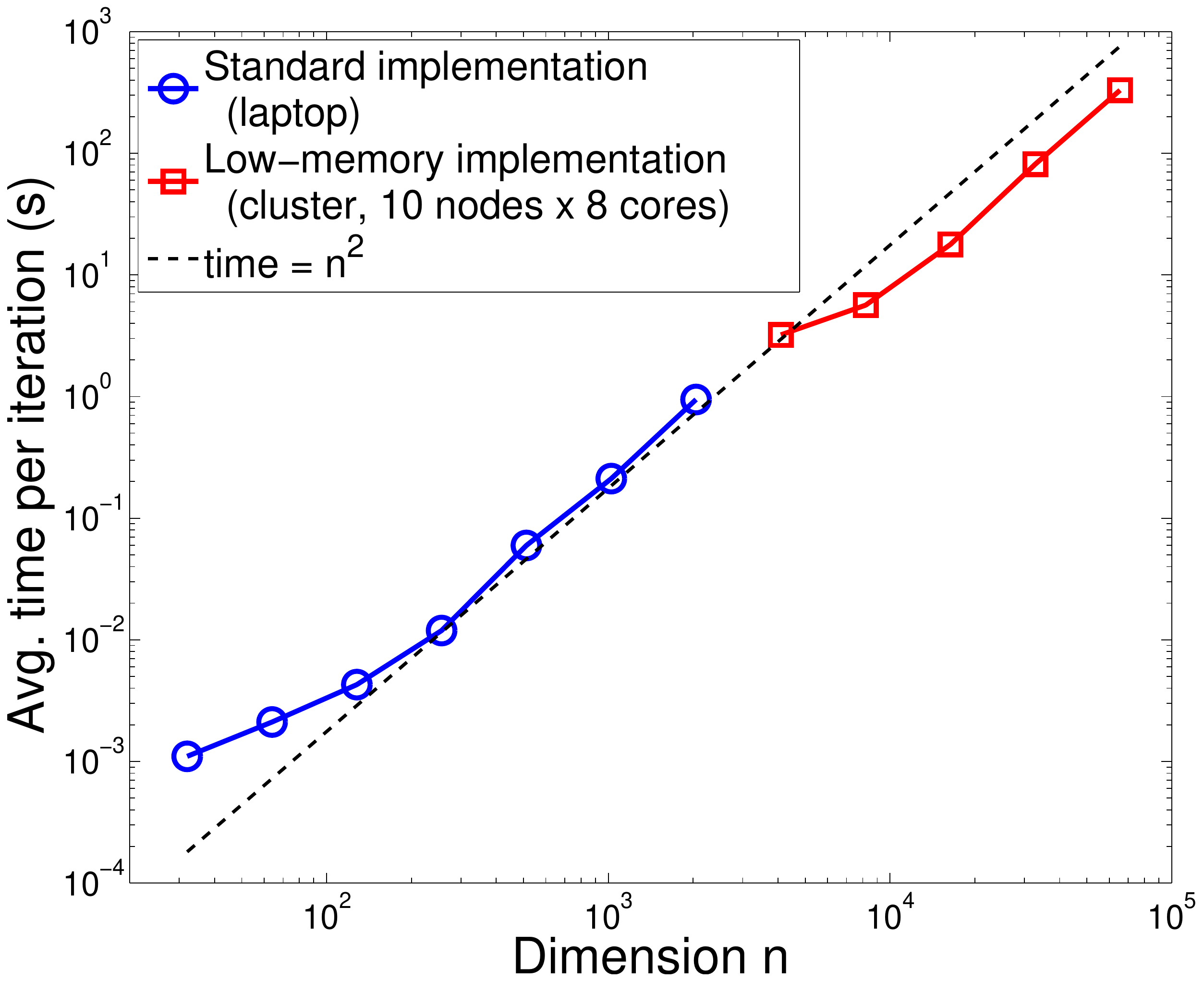}
    \caption{Mean time of 10 iterations: this includes the matrix multiplications as well as eigenvalue computations.
    (Left) shows times for a complete iteration of our method on a single computer using sparse matrix multiplies (``full memory'') and, above $11$ qubits, the custom low-memory implementation as well (not multi-threaded) on the same computer. (Right) shows times for just the \SVDfcn/\EIGfcn. 
    }
    \label{fig:iterSpeed}
\end{figure}

\begin{figure}[ht]
    \centering
    \newlength{\mySubfigSize}
    \setlength{\mySubfigSize}{4.7cm}
    \includegraphics[width=\mySubfigSize]{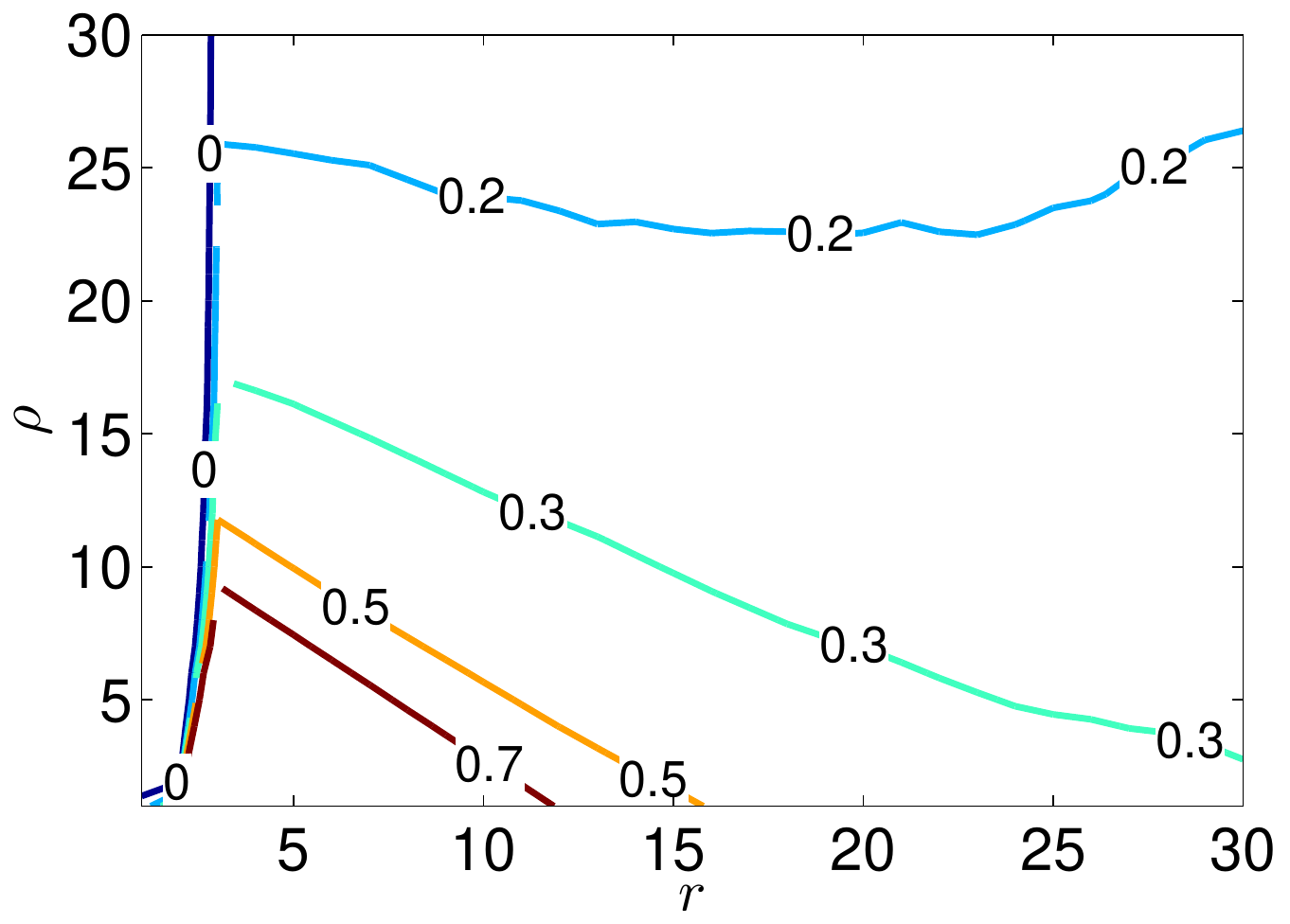}
    \includegraphics[width=\mySubfigSize]{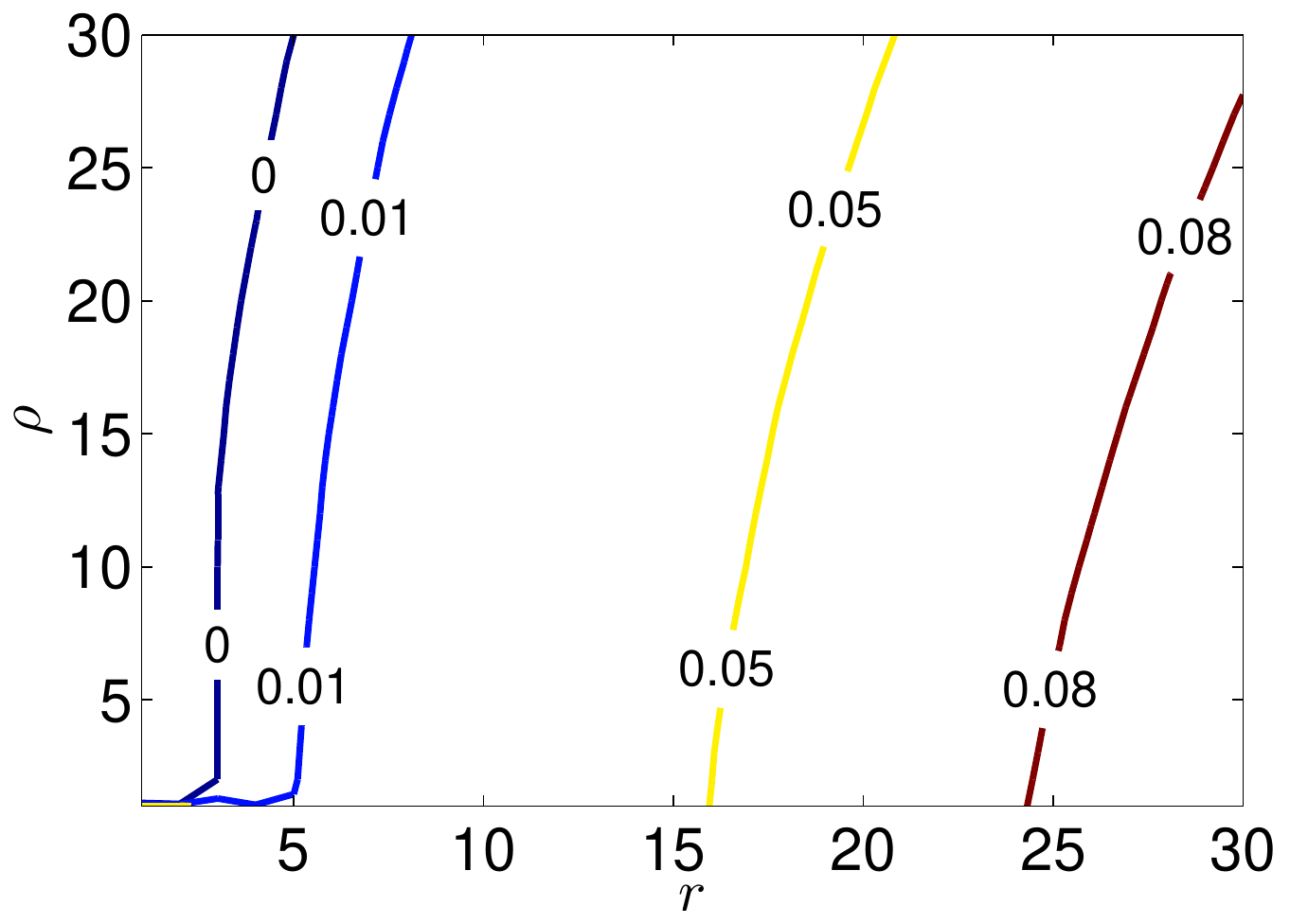} \\
    \includegraphics[width=\mySubfigSize]{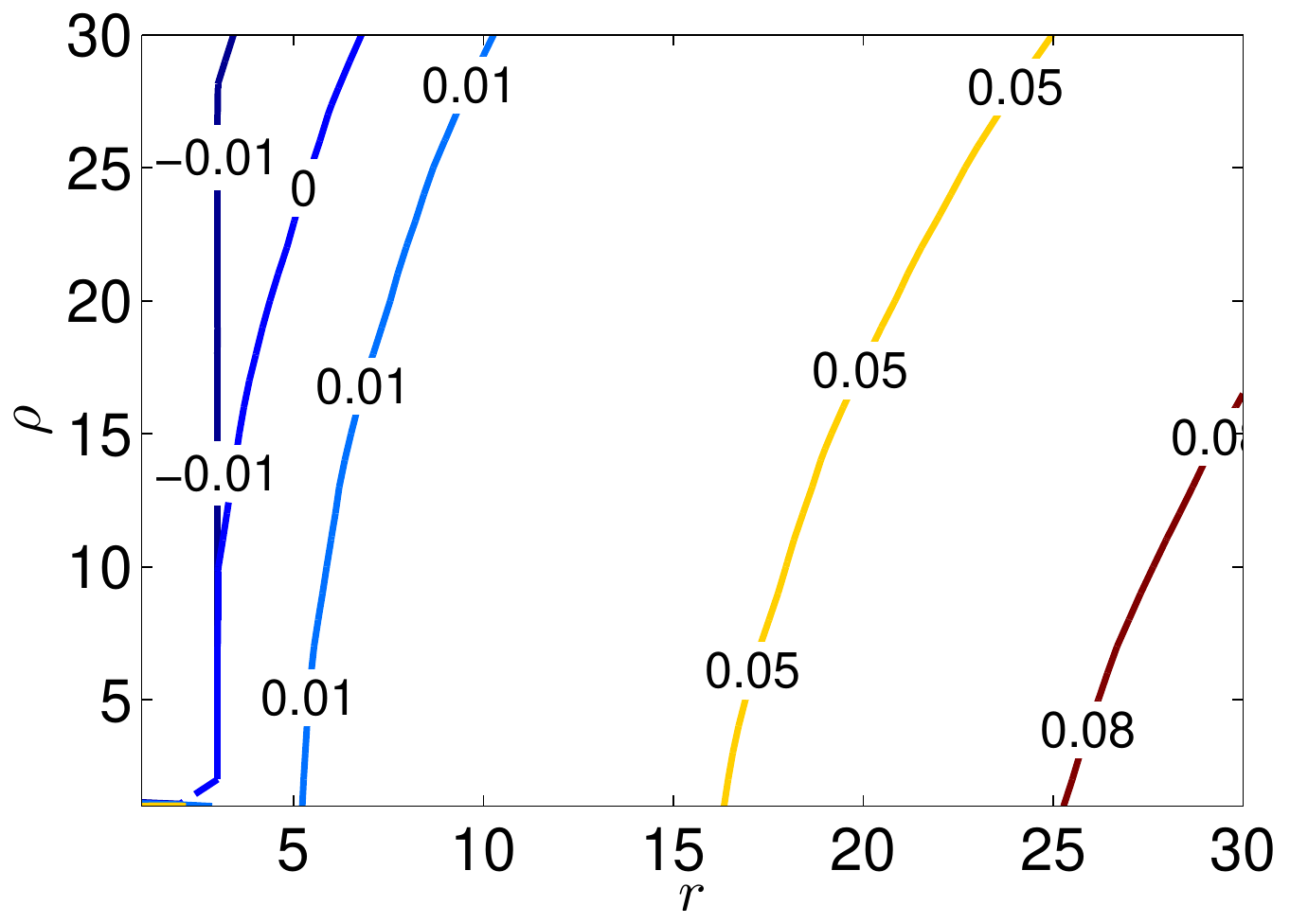}
    \includegraphics[width=\mySubfigSize]{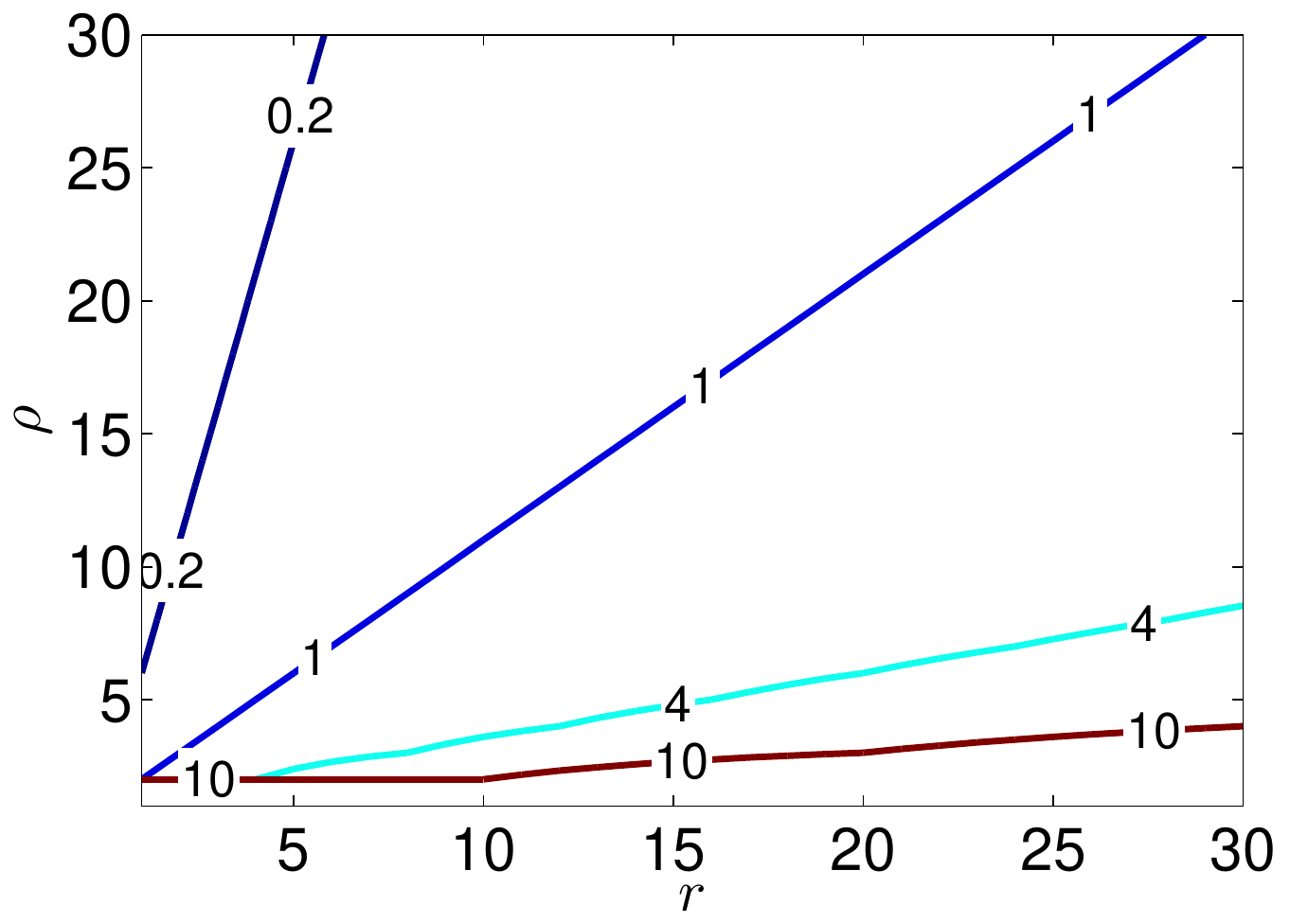}
    \caption{Top row: $\widetilde{\epsilon}$ for (left) $q=0$ and (right) $q=1$ power iterations. Bottom row: $\widetilde{\epsilon}$ for $q=2$ power iterations (left), and (right) shows the bound $\epsilon$.}
    \label{fig:epsilon}
\end{figure}

Figure~\ref{fig:convergence} (left) plots convergence and accuracy results for a quantum tomography problem with 8 qubits and $\numsam=4rn$ with $r=1$. The SVP algorithm works well on noisy problems but we focus here on a noiseless (and truly low-rank) problem in order to examine the effects of approximate SVD/eigenvalue computations. The figure shows that the power method with $q\ge 1$ is extremely effective even though it lacks theoretical guarantees; without the power method, take $\rho \simeq 20$ and we see convergence, albeit slower. When $\numsam$ is smaller and the R-RIP is not satisfied, taking $\rho$ or $q$ too small can lead to non-convergence.

Figure~\ref{fig:convergence} (right) 
is a direct comparison of $\EIGfcn$ (with $\rho=5$ and $q=3$) and the Lanczos method for multiplies of the type encountered in the algorithm. The $\EIGfcn$ has the same asymptotic complexity but much better constants.

Figure~\ref{fig:iterSpeed} shows that because the eigenvalue decomposition is a significant portion of the computational cost, using $\EIGfcn$ instead of Lanczos makes a difference. The difference is not pronounced in the small-scale full-memory implementation because the variable $\X$ is explicitly formed and matrix multiplies are relatively cheap compared to other operations in the code. For larger dimensions with the low-memory code, $\X$ is never explicitly formed and multiplying with the gradient is quite costly. The randomized method requires fewer multiplies, explaining its benefit. For 12 qubits, the Lanczos method averages 98.4 seconds/iteration, whereas the randomized method averages just 59.2 seconds.
The right subfigure shows that the low-memory implementation (which has memory requirement $\order(rn)$) still has only $\order(n^2)$ time complexity per iteration. 

Figure~\ref{fig:epsilon} tests Theorem~\ref{thm:Joel} by plotting the value of 
$$ \widetilde{\epsilon} = \|\X-\widetilde{\X}\|_F^2/\|\X-\X_r\|_F^2 - 1 $$
(which is bounded by $\epsilon$) for matrices $\X$ that are generated by the iterates of the algorithm. The algorithm is set for $r=1$ (so $\X$ is the sum of a rank 2 term, which includes the Nesterov term, and the full rank gradient), but the plots consider a range of $r$ and a range of oversampling parameters $\rho$. The plots use $q=0,1$ (top row, left to right) and $q=2$ (bottom row, left) power iterations. 
Because $\widetilde{\X}$ has rank $\ell=r+\rho$, it is possible for $\widetilde{\epsilon} < 0$, as we observe in the plots when $r$ is small and $\rho$ is large. For two power iterations, the error is excellent. In all cases, the observed error $\widetilde{\epsilon}$ is much better than the bound $\epsilon$ (shown bottom row, right) from Theorem~\ref{thm:Joel}, suggesting that it may be possible to have a more refined analysis.

Finally, to test scaling to very large data, we compute a 16 qubit state ($n=65536$), using a known quantum state as input, with realistic quantum mechanical perturbations (global depolarizing noise of level $\gamma=0.01$; see \cite{flammia2012quantum}) as well as AWGN to give a SNR of 30~dB, and $\numsam=5n=327680$ measurements. The first iteration uses Lanczos and all subsequent iterations use $\EIGfcn$ using $\rho=5$ and $q=3$ power iterations. On a cluster with 10 computers, the mean time per iteration is $401$ seconds.
The table in Fig.~\ref{fig:16} (left) 
shows the error metrics of the recovered matrix, and Fig.~\ref{fig:16} (right) 
plots the convergence rate of the Frobenius-norm error and trace distance. 

\begin{figure}[t]
    \centering
\begin{tabular}[b]{llll}  
\toprule
  & Trace distance  & Fidelity  & \\
$\|\x-\x^\star\|_F$	& $\|\X-\X^\star\|_*$ & $F(\X,\X^\star) $ & $F(\X,\X^\star)^2$ \\
\midrule
0.0256 &   0.0363  &  0.9998  &  0.9997 \\ 
 \bottomrule
\end{tabular} 
  \raisebox{-0.35\height}{ 
  \includegraphics[width=2.3in]{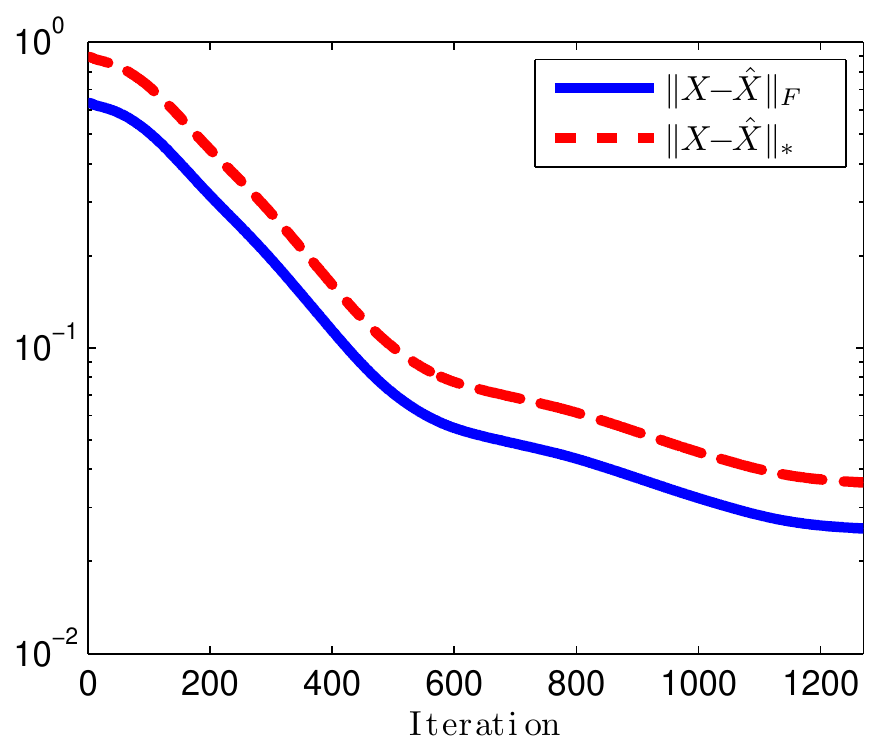}
  }
  \caption{ The table (left) 
  shows error metrics for the noisy rank-1 16-qubit recovery. The figure (right) 
  shows the convergence rate for the same simulation.} 
  \label{fig:16}
\end{figure}

\begin{figure}[t]
    \centering
    \includegraphics[width=.49\columnwidth]{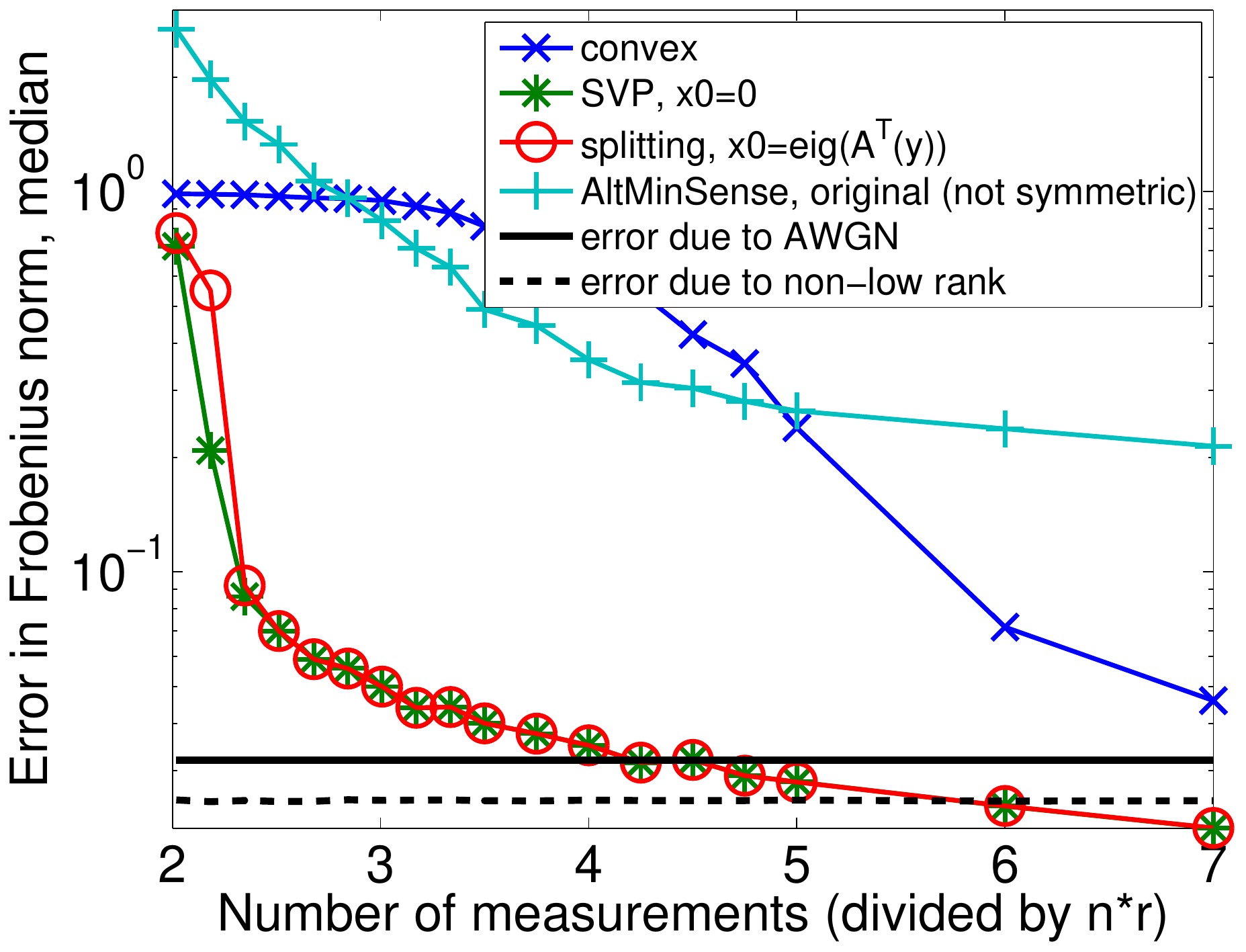}
	\caption{Accuracy comparison of several algorithms, as a function of number of samples $\numsam$. Each point is the median of the results of 20 simulations.}
	\label{fig:accuracyComparisons}
\end{figure}
Figure~\ref{fig:accuracyComparisons} reports the median error on 20 test problems across a range of $\numsam$. Here, $\x^\star$ is only approximately low rank and $y$ is contaminated with noise. We compare the convex approach~\cite{flammia2012quantum}, the ``AltMinSense'' approach~\cite{alternatingMinRIP}, and a standard splitting approach. AltMinSense and the convex approach have poor accuracy; the accuracy of AltMinSense can be improved by incorporating symmetry, but this changes the algorithm fundamentally and the theoretical guarantees are lost. The splitting approach, if initialized correctly, is accurate, but lacks guarantees. Furthermore, it is slower in practice due to slower convergence, though for some simple problems (i.e., no convex constraints $\CC$) it is possible to accelerate using L-BFGS~\cite{FrankWolfeOMP}.

\section{Conclusion}
Randomization is a powerful tool to accelerate and scale optimization algorithms, and it can be rigorously included in algorithms that are robust to small errors.
In this paper, we leverage randomized approximations to remove memory bottlenecks by merging the two-key steps of most recovery algorithms in affine rank minimization problems: gradient calculation and low-rank projection. Unfortunately, the current black-box approximation guarantees, such as Theorem~\ref{thm:Joel}, are too pessimistic to be directly used in theoretical characterizations of our approach. For future work, motivated by the overwhelming empirical evidence of the good performance of our approach, 
we plan to directly analyze the impact of randomization in characterizing the algorithmic performance.

\section*{Acknowledgment}
VC and AK's work was supported in part by the European Commission under Grant MIRG-268398, ERC Future Proof, SNF 200021-132548, and ARO MURI W911NF0910383. 
SRB is supported by the Fondation Sciences Math\'ematiques de Paris. The authors thank Alex Gittens for his insightful comments and Yi-Kai Liu and Steve Flammia for helpful discussions.

\appendix
\section{Proofs}
\newcommand{\wx}{\widetilde{\x}}
\newcommand{\el}{\ell} 
\newcommand{\dimension}{m \times n}
\newcommand{\nablaS}{\nabla_{\mathcal{S}}}
\newcommand{\projS}{\proj_{\mathcal{S}}}
\begin{proof}[Proof of Theorem~\ref{thm: invariant}]
There are three aspects to the proof. Even without approximate SVD calculations, the problem is non-convex, so we must leverage the R-RIP to prove that iterates converge. Mixed in with this calculation is the approximate nature of our rank $\ell$ point $\wx_{i+1}$, where we will apply the bounds from Theorem~\ref{thm:Joel}. Finally, we relate $\wx_{i+1}$ to its rank $r$ version $\x_{i+1}$.

An important definition for our subsequent developments is the following:
\begin{definition}[$\epsilon $-approximate low-rank projection] \label{def:appr_svd}
Let $ \X $ be an arbitrary matrix. For any $\epsilon > 0$, $ \mathcal{P}_{r',\el'}^{\epsilon}(\X) $ provides a rank-$\el'$ matrix approximation to $\X$ such that
\begin{align}
\E \vectornormbig{\mathcal{P}_{r',\el'}^{\epsilon}(\X) - \X}_F^2 \leq (1 + \epsilon) \vectornormbig{\mathcal{P}_{r'}(\X) - \X}_F^2, \label{eq:appr_svd:00}
\end{align} where $ \mathcal{P}_{r'}(\X) \in \argmin_{\Y: \rank(\Y) \leq r'} \vectornorm{\X - \Y}_F $.
\end{definition}

\newcommand*\xbar[1]{%
  \hbox{%
    \vbox{%
      \hrule height 0.5pt 
      \kern0.5ex
      \hbox{%
        \kern-0.1em
        \ensuremath{#1}%
        \kern-0.1em
      }%
    }%
  }%
} 

\newcommand{\StrongConvexity}{M} 
Let $\x_i$ be the putative rank $r$ solution at the $i$-th iteration, $\x^\star$ be the rank $r$ matrix we are looking for and $\wx_{i+1}$ be the rank $l$ matrix, obtained using approximate SVD calculations. Define $L := 2(1+\delta_{r + \el})$ and $\StrongConvexity := 2(1-\delta_{2r})$. 
Then, we have:
\begin{align}
f(\wx_{i+1}) &= f(\x_{i}) + \langle \nabla f(\x_i),~\wx_{i+1} - \x_i\rangle + \| \linmap( \wx_{i+1} - \x_i ) \|_F^2 \nonumber \\
                 &\leq f(\x_{i}) + \langle \nabla f(\x_i),~\wx_{i+1} - \x_i\rangle + \frac{L}{2}\| \wx_{i+1} - \x_i \|_F^2 \nonumber \\
				 &=f(\x_{i}) -\frac{1}{2L}\|\nabla  f(\x_i)\|_F^2 + \frac{L}{2}\left(\|\wx_{i+1} - \x_i \|_F^2 + 2\langle \frac{1}{L}\nabla  f(\x_i),~\wx_{i+1} - \x_i\rangle + \frac{1}{L^2}\|\nabla f(\x_i)\|_F^2\right) \nonumber \\
				 &=f(\x_{i}) -\frac{1}{2L}\|\nabla  f(\x_i)\|_F^2 + \frac{L}{2} \|\wx_{i+1} - \left(\x_i - \frac{1}{L}\nabla  f(\x_i)\right)\|_F^2. \label{eq:01}
             \end{align} By construction $\wx_{i+1} \in \mathcal{P}_{r,\el}^{\epsilon}\left( \x_i - \frac{1}{L}\nabla  f(\x_i)\right)$ (since the step-size is $\mu=1/L$), so, for $\xbar{\x}_{i+1} \in \mathcal{P}_{r}\left( \x_i - \frac{1}{L}\nabla  f(\x_i)\right)$, 
\begin{align}
\E  \|\wx_{i+1} - (\x_i - \frac{1}{L}\nabla  f(\x_i))\|_F^2 &\leq (1+\epsilon) \|\xbar{\x}_{i+1} - (\x_i - \frac{1}{L}\nabla  f(\x_i))\|_F^2 \nonumber \\
& \le (1+\epsilon)\|\x^{\star} - (\x_i - \frac{1}{L}\nabla  f(\x_i))\|_F^2 \label{eq:02} 
\end{align}
by the definition of $\mathcal{P}_{r}(\cdot)$ (since $\rank(\x^{\star}) = r$). 
Combining \eqref{eq:02} with \eqref{eq:01}, we obtain:
\begin{align}
\E f(\wx_{i+1}) &\leq f(\x_{i}) - \frac{1}{2L} \|\nabla f(\x_i)\|_F^2 + \frac{L}{2} (1+\epsilon) \| \x^{\star} - \x_i + \frac{1}{L}\nabla  f(\x_i) \|_F^2 \nonumber \\
		&=f(\x_{i}) - \frac{1}{2L} \|\nabla  f(\x_i)\|_F^2 + (1+\epsilon)\left( \frac{1}{2L}  \|\nabla  f(\x_i)\|_F^2 + \langle \nabla  f(\x_i),~\x^{\star} - \x_i\rangle + \frac{L}{2}\| \x^{\star} - \x_i\|_F^2 \right) \nonumber \\
        &\leq (1+\epsilon)\left[f(\x_{i}) + \langle \nabla  f(\x_i),~\x^{\star} - \x_i\rangle + \frac{L}{2}\| \x^{\star} - \x_i\|_F^2\right] + \frac{\epsilon}{2L} \|\nabla  f(\x_i) \|_F^2  \label{eq:90}
\end{align} where we use the fact that $f(\x_i)\ge 0$ in the last inequality. 
Due to the restricted strong convexity of $f$ that follows from the restricted isometry property, 
we have:
\begin{align}
f(\x^{\star}) &\geq f(\x_i) + \langle \nabla  f(\x_i), \x^{\star} - \x_i\rangle + \frac{\StrongConvexity}{2}\|\x^{\star} - \x_i\|_F^2 \nonumber \\
f(\x^{\star}) - \frac{\StrongConvexity}{2}\|\x^{\star} - \x_i\|_F^2 &\geq f(\x_i) + \langle \nabla  f(\x_i), \x^{\star} - \x_i\rangle \nonumber
\end{align} which, combined with \eqref{eq:90}, leads to:
\begin{align}
\E f(\wx_{i+1}) &\leq (1+\epsilon)\left[f(\x^{\star}) + \frac{L - \StrongConvexity}{2}\| \x^{\star} - \x_i\|_F^2\right] + \frac{\epsilon}{2L} \|\nabla f(\x_i) \|_F^2 \label{eq:03}
\end{align}

Due to the R-RIP,  
\begin{align}
    \|\x^{\star} - \x_i\|_F^2 \leq \frac{\|\sensing  (\x^{\star} - \x_i)\|_2^2}{1-\delta_{2r}} \label{eq:04} 
\end{align} 
Now define a constant $C$ and assume $f(\x_i)  = \|\obs - \sensing \x_i\|_2^2 > C^2\|\noise\|_2^2$ (if the assumption fails, it means $\x_i$ is already close to $\X^\star$). In particular, in the noiseless case $\|\noise\|=0$, we may pick $C$ arbitrarily large and set all $1/C$ terms to zero. 
\begin{align}
\|\sensing(\x^{\star} - \x_i)\|_F^2 &= \|\obs - \sensing(\x_i) - \noise\|_2^2 \nonumber \\ 
										  &= \|\obs - \sensing(\x_i)\|_2^2 + \|\noise\|_2^2 - 2\langle \noise, \obs - \sensing(\x_i) \rangle \nonumber \\
			  							&\leq f(\x_i) + \|\noise\|_2^2 + 2\|\noise\|_2 \|\obs - \sensing(\x_i)\|_2 \nonumber \\
										&\leq f(\x_i) + \|\noise\|_2^2 + \frac{2}{C}f(\x_i) \label{eq:05}
\end{align}
Substituting \eqref{eq:05} and \eqref{eq:04} into \eqref{eq:03}, expanding the values of $L$ and $M$, and noting that 
$f(\x^{\star}) = \| \obs - \sensing (\x^{\star}) \|_2^2 = \|\noise \|_2^2$, gives
\begin{align}
\E f(\wx_{i+1}) &\leq (1+\epsilon)\left[\|\noise\|_2^2 + \frac{\delta_{r+\el} + \delta_{2r}}{1-\delta_{2r}}\left(f(\x_i) + \|\noise\|_2^2 + \frac{2}{C}f(\x_i)\right)\right] + \frac{\epsilon}{2L} \|\nabla f(\x_i) \|_F^2 \\ 
				 &\leq (1+\epsilon)\left[\frac{\delta_{r+\el} + \delta_{2r}}{1-\delta_{2r}}\left(1+\frac{2}{C}\right)f(\x_i) + \left(1 + \frac{\delta_{r+\el} + \delta_{2r}}{1-\delta_{2r}}\right)\|\noise\|_2^2\right] + \frac{\epsilon}{2L} \|\nabla f(\x_i) \|_F^2 \label{eq:previousResult1}
\end{align}
We bound $\|\nabla f(\x_i)\|$ using our assumption on the magnitude of $\|\sensing\|$:
\begin{equation} \label{eq:new_result}
\|\nabla f(\x_i) \|_F^2 
            = 4 \|\sensing^{\ast}\left (\obs - \sensing(\x_i)\right)\|_F^2 
            \leq 4 \|\sensing^{\ast}\|^2 \| \obs - \sensing(\x_i) \|_2^2 
			= 4 \|\sensing\|^2 f(\x_i)
            \leq 4 \frac{mn}{\numsam} f(\x_i)
\end{equation}
For quantum tomography, we even have $\sensing \sensing^{\ast} = \frac{mn}{\numsam} \mathcal{I}$, so the inequality holds with equality (and $m=n$).

Combining \eqref{eq:previousResult1} with \eqref{eq:new_result} and by the definition of $L$, we obtain:
\begin{align}
\E f(\wx_{i+1}) &\leq (1+\epsilon)\left[\frac{\delta_{r+\el} + \delta_{2r}}{1-\delta_{2r}}\left(1+\frac{2}{C}\right)f(\x_i) + \left(1 + \frac{\delta_{r+\el} + \delta_{2r}}{1-\delta_{2r}}\right)\|\noise\|_2^2\right] + \frac{\epsilon}{1 + \delta_{r + \ell}}\cdot \frac{mn}{\numsam}  f(\x_i) \\ 
                 &= \underbrace{\left( \frac{\epsilon}{1 + \delta_{r + \ell}}\cdot \frac{mn}{\numsam} + (1+\epsilon)\frac{\delta_{r+\el} + \delta_{2r}}{1-\delta_{2r}}\left(1+\frac{2}{C}\right)\right)}_{\theta'}f(\x_i)  + 
                 \underbrace{(1+\epsilon)\left(1 + \frac{\delta_{r+\el} + \delta_{2r}}{1-\delta_{2r}}\right)}_{\tau'}\|\noise\|_2^2  \label{eq:previousResult}
\end{align}

Note that if an exact SVD computation is used, then not only is $\epsilon=0$ but also $\wx_{i+1}$ is rank $r$, so we are done and can use $\theta=\theta'$ and $\tau=\tau'$.   To finish the proof, we now relate $\E f(\x_{i+1})$ to $\E f(\wx_{i+1})$. In the algorithm, $\x_{i+1}$ is the output of \SVDfcn, and $\wx_{i+1}$ is the intermediate value $U\Sigma V^\adj$ on line 10 of Algo.~\ref{algo:rankProjection}. Given $\wx_{i+1}$ with $\rank(\wx_{i+1}) = \el > r$, $\x_{i+1}$ is defined as the best rank-$r$ approximation to $\wx_{i+1}$.\footnote{%
\newcommand{\px}{\proj_r(\wx_{i+1})}%
If we include a convex constraint $\CC$ then instead of defining $\x_{i+1} = \px$ we have $\x_{i+1} = \proj_C( \px )$.
In this case, $$\|  \proj_C( \px ) - \x^\star\|_F = \|  \proj_C( \px  - \x^\star )\|_F \le \| \px - \x^\star\|_F.$$ The first equality follows from $\x^\star \in \CC$ and the second is true 
since the projection onto a non-empty closed convex set is non-expansive. Hence the result in \eqref{eq:triangle_ineq} still applies when we include the $\CC$ constraints.
}
Thus, the following inequality holds true:
\begin{align}
\|\x_{i+1} - \x^{\star}\|_F &= \|\x_{i+1} - \wx_{i+1} + \wx_{i+1} - \x^{\star}\|_F \nonumber \\ 
										  &\leq \|\x_{i+1} - \wx_{i+1}\|_F + \|\wx_{i+1} - \x^{\star}\|_F \nonumber \\ 
										  &\leq 2\|\wx_{i+1} - \x^{\star}\|_F \label{eq:triangle_ineq}
\end{align} since $\|\x_{i+1} - \wx_{i+1}\|_F \leq \|\x^{\star} - \wx_{i+1}\|_F $. 
In particular, since the above is valid for any value of the random variable $\wx_{i+1}$, 
$\E\; \|\x_{i+1} - \x^{\star}\|_F^2 \le   \E \; 4\|\wx_{i+1} - \x^{\star}\|_F^2 $.
This bound is pessimistic and in practice the constant is close to 1 rather than 4.

We will again assume that $f(\wx_{i+1}),f(\x_{i+1}) \ge C^2\|\noise\|_2^2$, and $C > 2$, since otherwise the current point is a good-enough solution.  We have:
\begin{align*}
f(\x_{i+1}) = \|\obs - \sensing (\x_{i+1})\|_2^2 &= \|\sensing(\x^{\star} - \x_{i+1}) + \noise \|_2^2 \\
     &= \|\sensing(\x^{\star} - \x_{i+1})\|_2^2 + \|\noise \|_2^2 + 2\langle \sensing (\x^{\star} - \x_{i+1}), \noise \rangle  \\
     &= \|\sensing(\x^{\star} - \x_{i+1})\|_2^2 + \|\noise \|_2^2 + 2\langle \obs - \sensing (\x_{i+1}) - \noise, \noise \rangle  \\
     &= \|\sensing(\x^{\star} - \x_{i+1})\|_2^2 + \|\noise \|_2^2 + 2\langle \obs - \sensing (\x_{i+1}), \noise \rangle + 2\langle -\noise, \noise \rangle  \\
     &\leq \|\sensing(\x^{\star} - \x_{i+1})\|_2^2 + \|\noise \|_2^2 + 2\| \obs - \sensing (\x_{i+1})\|_2 \|\noise \|_2 - 2\| \noise\|_2^2  \\
     &\leq \|\sensing(\x^{\star} - \x_{i+1})\|_2^2 -\|\noise \|_2^2 + \frac{2}{C}f(\x_{i+1}) 
\end{align*}
    which, if $1 - 2/C \ge 0 $, implies
\begin{equation}
f(\x_{i+1}) \leq \frac{1}{1-2/C}\|\sensing(\x^{\star} - \x_{i+1})\|_2^2 - \frac{1}{1-2/C}\|\noise\|_2^2 \label{eq:brr}
\end{equation} By the R-RIP assumption, we have:
\begin{align}
\|\sensing(\x^{\star} - \x_{i+1})\|_2^2 \leq (1+\delta_{2r}) \|\x^{\star} - \x_{i+1}\|_F^2. \label{eq:RIPagain}
\end{align} Using \eqref{eq:triangle_ineq} and \eqref{eq:RIPagain} in \eqref{eq:brr}, we obtain:
\begin{align}
f(\x_{i+1}) &\leq \frac{4(1+\delta_{2r})}{1-2/C}\|\wx_{i+1} - \x^{\star}\|_F^2 - \frac{1}{1-2/C}\|\noise\|_2^2 \label{eq:eq00}
\end{align} Using the R-RIP property again, the following sequence of inequalities holds:
\begin{align}
\|\wx_{i+1} - \x^{\star}\|_F^2 &\leq \frac{\|\sensing(\wx_{i+1} - \x^{\star})\|_F^2}{1-\delta_{r + \el}} \nonumber \\
																 &\leq \frac{1 + 2/C}{1 - \delta_{r + \el}}f(\wx_{i+1}) + \frac{1}{1-\delta_{r + \el}}\|\noise\|_2^2 \label{eq:eq01}
\end{align} where the second inequality is obtained following same motions as \eqref{eq:05}. Combining \eqref{eq:eq00}-\eqref{eq:eq01} with \eqref{eq:previousResult}, we obtain:
\begin{align}
\E f(\x_{i+1}) &\leq \underbrace{\frac{4(1+\delta_{2r})}{1-2/C}\cdot\frac{1 + 2/C}{1 - \delta_{r + \el}}\cdot
\theta'}_{\theta} \cdot
f(\x_i) \nonumber 
				 + \underbrace{ \left( 
        \frac{4(1+\delta_{2r})}{1-2/C} \cdot \frac{1 + 2/C}{1 - \delta_{r + \el}} \cdot 
        \tau'
        + \frac{4(1+\delta_{2r})}{1-2/C}\cdot\frac{1}{1 - \delta_{r + \el}} - \frac{1}{1-2/C} \right)}_\tau \|\noise\|_2^2
\end{align}

Now we simplify the result to make it more interpretable. Define $\rho = \el - r$. Let $c$ be the smallest integer such that $\ell\ge(c-1)r$ (and for simplicity, assume $\ell=(c-1)r$)
so that $\delta_{r+\el} = \delta_{cr}$
and $\delta_{r+\el} + \delta_{2r} \le 2 \delta_{cr}$. 
By Theorem~\ref{thm:Joel}, $\epsilon \le \frac{r}{\rho-1}=\frac{r}{ (c-2)r -1 }$.
For concreteness, take $C\ge 4$ so that $1+2/C \le 3/2$ and $(1-2/C)^{-1} \le 2$.
Then
\begin{equation}
    \theta \le 12 \cdot \frac{1+\delta_{2r}}{1 - \delta_{cr}} \cdot 
\left( 
    \frac{\epsilon}{1+\delta_{cr}}\cdot\frac{m n }{\numsam} + (1+\epsilon)\frac{3\delta_{cr}}{1-\delta_{2r}} 
\right)
\end{equation}
and
\begin{align}
\tau &\le 
 \left( 
 12 \cdot \frac{1+\delta_{2r}}{1 - \delta_{cr}}  \cdot  (1+\epsilon)  \left(1 + \frac{\delta_{2r} + \delta_{cr}}{1-\delta_{2r}}\right) 
 + \frac{8(1+\delta_{2r})}{1 - \delta_{cr}} 
 \right)  \nonumber \\
 &\le  
  \frac{1+\delta_{2r}}{1 - \delta_{cr}}  \cdot  
 \left( 
 12\cdot(1+\epsilon)  \left(1 + \frac{2\delta_{cr}}{1-\delta_{2r}}\right) 
 + 8
 \right)  
 \end{align}

\end{proof}

\bibliographystyle{amsalpha}
\bibliography{rsvp}

\end{document}